\documentclass[review]{elsarticle}

\usepackage{amssymb}
\usepackage{graphicx}
\usepackage{amsmath}
\usepackage{lineno,hyperref}
\usepackage{tikz}
\usepackage{amsfonts}
\usepackage{amsfonts}
\usepackage{lineno,hyperref}
\usepackage{filecontents}
\usepackage{lineno,hyperref}
\usepackage{amsmath,amssymb,amscd,amsthm,amstext}

\newtheorem{theorem}{Theorem}
\newtheorem{conjecture}{Conjecture}
\newtheorem{claim}{Claim}
\newtheorem{example}{Example}
\newtheorem{ob}{Observation}

\everymath{\displaystyle}

\modulolinenumbers[5]

\makeatletter
\def\ps@pprintTitle{%
  \let\@oddhead\@empty
  \let\@evenhead\@empty
 \let\@oddfoot\@empty
  \let\@evenfoot\@oddfoot
}
\makeatother

\begin{document}

\begin{frontmatter}
\title{\textbf{On the path partition of graphs }}
\author{\textbf{Mekkia Kouider} }
\address{Universit\'{e} Paris Sud, France.}
\ead{ km@lri.fr}
\author {\textbf{Mohamed Zamime}}
\address{Faculty of Technology, University of Medea, Algeria}  
\ead{zamimemohamed@yahoo.com}
\begin{abstract}
  Let $G$ be a graph of order $n$. The maximum and minimum degree of $G$ are denoted by
  $\Delta$ and $\delta$ respectively.
  The \emph{path partition number} $\mu (G)$ of a graph $G$ is the minimum number of
paths needed to partition   the vertices of $G$. Magnant, Wang and Yuan
conjectured that $\mu (G)\leq \max \left \{ \frac{n}{\delta +1},
\frac{\left( \Delta -\delta \right) n}{\left( \Delta +\delta \right) }\right \} .$
In this work, we give a positive answer to this conjecture, for
$ \Delta \geq 2 \delta $.\medskip
\end{abstract}

\end{frontmatter}

\section{Introduction}

Throughout the paper, all graphs are finite, simple and undirected. Let $G$
be a graph with vertex-set $V(G)$ and edge-set $E(G).$ We denote by $n$ the
order of $G.$ The \emph{neighborhood} of a vertex $v\in V$ is $N\left(
v\right) =\left\{ u\in V:uv\in E\right\} .$ The \emph{degree }of $v,$
denoted by $d\left( v\right) ,$ is the size of its neighborhood. The \emph{%
minimum degree} of the graph $G$ is denoted by $\delta (G)$, and the \emph{%
maximum degree} by $\Delta (G).$

Let $A$ and $B$ be two subsets of $V(G)$. Let $\varepsilon (A,B)$ be the
number of edges with one end vertex in the set $A$ the other one in the set $%
B$.

In this work, we deal with the partition problem. The cover
problem and the partition problem constitute a large and important class of
well studied problems in the fields of graph theory. A \textit{cycle cover}
of a graph (resp. a \textit{path cover}) is a set $\mathcal{C}$ of cycles
(resp. paths) of the graph such that each vertex belongs to at least one
cycle (resp. one path) of $\mathcal{C}$. Many results on these concepts,
have been given in the literature. For example, Kouider \cite{Kouider 1}\cite%
{k3}, and Kouider and Lonc \cite{k4} studied the problem of covering a graph
by a minimum number of cycles. More details and references can be found in
the survey of Manuel \cite{Man}.

Among the many variations of the partition problem, we mention the \emph{%
path partition} that has been studied intensively for about sixty years. A
family $\mathcal{P}$ of paths is called a \emph{path partition} of a graph $%
G $ if its members cover the vertices of the graph and are vertex disjoint.
Its cardinality $\left\vert \mathcal{P}\right\vert $ is the number of paths
of $\mathcal{P}.$ The \emph{path partition number} of $G$ is $\mu (G)=\min
\{\left\vert \mathcal{P}\right\vert :\mathcal{P}$ is a path partition of $%
G\}.$ The concept of path partition number was introduced by Ore \cite{Ore}
in $1961$. Several works have been done in this topic. See for example \cite%
{M.Chen, H.Enomoto, Han}.

In $1996$, Reed proved the folowing result \cite{B Reed}.

\begin{theorem}
\cite{B Reed} Let $G$ be a connected cubic graph on $n$ vertices. Then%
\newline
$\mu (G)\leq \left\lceil \frac{n}{9}\right\rceil .$
\end{theorem}

Furthermore, for $2$-connected graphs, a better bound is established by Yu%
\cite{G. Yu}.

\begin{theorem}
Let $G$ be a $2$-connected cubic graph on $n$ vertices. Then \newline
$\mu (G)\leq \left\lceil \frac{n}{10}\right\rceil .$
\end{theorem}

For regular graphs, in $2009$, Magnant and Martin \cite{Mag Mart}
conjectured the following.

\begin{conjecture}
\label{conj0}\label{conj1}\label{conj2}\cite{Mag Mart} Let $G$ be a $d$%
-regular graph on $n$ vertices. \newline
Then $\mu (G)\leq \frac{n}{d+1}.$
\end{conjecture}

They verified this last conjecture for the case $d\leq 5$ (see \cite{Mag
Mart}). In $2018$, Han has an asymptotic answer.

\begin{theorem}
\cite{Han} For every $c,$ $0<c<1$ and $\alpha >0 $, there exists $n_{0}$
such that if $n\geq n_{0}$, $d\geq cn$ and $G$ is a $d$-regular graph on $n$
vertices, then $n/(d+1)$ vertex-disjoint paths cover all vertices of $G$
except $\alpha n$.
\end{theorem}

Gruskys and Letzter \cite{GL} improve this result by allowing to take $%
\alpha =0.$

In $2016$, Magnant, Wang and Yuan\cite{Mag Wan} extend Conjecture \ref{conj0}
to general graphs as follows.

\begin{conjecture}
\label{conj 2}\label{conjx}\cite{Mag Wan} Let $G$ be a graph on $n$
vertices. Then 
\begin{equation*}
\mu (G)\leq \max \left\{ \dfrac{n}{\delta +1},\dfrac{\left( \Delta -\delta
\right) n}{\left( \Delta +\delta \right) }\right\} .
\end{equation*}
\end{conjecture}

\bigskip If true, the last conjecture would be sharp. For $%
\delta +2\leq \Delta ,$ the bound is achieved by the collection of disjoint
copies of $K_{\delta ,\Delta }.$ For $\delta =\Delta ,$ it is achieved by
the collection of disjoint copies of complete graphs $K_{\delta +1}.$ This
conjecture is proved in \cite{Mag Wan} for the case $\delta =1$ and $\delta
=2.$ \newline
In this work, we prove Conjecture \ref{conjx} for all graphs with maximum
degree $\Delta $ at least $2\delta .$

\begin{theorem}
\label{thm 1} Let $G$ be a graph of order $n$ of minimum degree $\delta
,(\delta \geq 2),$ and maximum degree $\Delta $ with $\Delta \geq 2\delta .$
Then $\mu (G)\leq \dfrac{\left( \Delta -\delta \right) n}{\left( \Delta
+\delta \right) }.$
\end{theorem}

We remark that $\dfrac{n}{\delta +1}\leq \dfrac{\left( \Delta -\delta
\right) n} {\left( \Delta +\delta \right) }$ if and only if $\delta +2\leq
\Delta .$ So for $\delta \geq 2$ and $\Delta \geq 2\delta $, the inequality
of the theorem is equivalent to

\begin{equation*}
\mu (G)\leq \max \left\{ \dfrac{n}{\delta +1},\dfrac{\left( \Delta -\delta
\right) n}{\left( \Delta +\delta \right) }\right\}
\end{equation*}%
which is the inequality of the Conjecture \ref{conjx}.

\section{Preliminaries}

 Let us introduce the following notations and definitions. Let%
\textrm{\ $\mathcal{P}$ }be a minimum path partition of $V(G)$. So, $%
\left\vert \mathcal{P}\right\vert =\mu (G)$. Let $p_{i}$ be the number of
paths of order $i\in \{1,2\}$ in $\mathcal{P}.$ \newline We may suppose
that $p_{1}+p_{2}\neq 0,$ otherwise we have $\mu (G)\leq \frac{n}{3}$. As $\Delta
\geq 2\delta $, we get $\mu (G)\leq \dfrac{\left( \Delta -\delta \right) n}{%
\left( \Delta +\delta \right) }$ and the problem is resolved.\newline
Let $V_{1}$ be the set of isolated vertices of $\mathcal{P}$ and $V_{2}$ be
the set of end vertices of the isolated edges of $\mathcal{P}$. We denote by 
$R$ any path in $\mathcal{P}$, and we write $R=R[a,b]=[a,...,b]$ if $a$ and $%
b$ are the end vertices of $R$. We set $End(R)={\{a,b\}.}$ Let $Int\left(
R\right) $ be the set of internal vertices of $R$. Let $%
\mathcal{A}$ $\subseteq {\mathcal{P}}$. We denote by $Int(\mathcal{A})$
(resp. $End(\mathcal{A})$) the set of internal (resp. end ) vertices of the
paths of $\mathcal{A}$. For $i$ fixed, we denote by $R_{i}$ any path of
order $i$. We set ${\mathcal{R}}_{i}$ the set of paths of order $i$. By $%
abcd$ or $[a,b,c,d] $ we denote a path with $4$ vertices. For $i$ odd, $%
i\geq 3$, let us set $C_{i}=\bigcup\limits_{R\in \mathcal{R}_{i}}c\left(
R\right) $,\text{\ }where $c\left( R\right) $ denotes the central vertex of
the path $R$.

\begin{example}$ $\newline
Let us illustrate the above notations relative to a partition in
Figure 1.\newline
\bigskip 
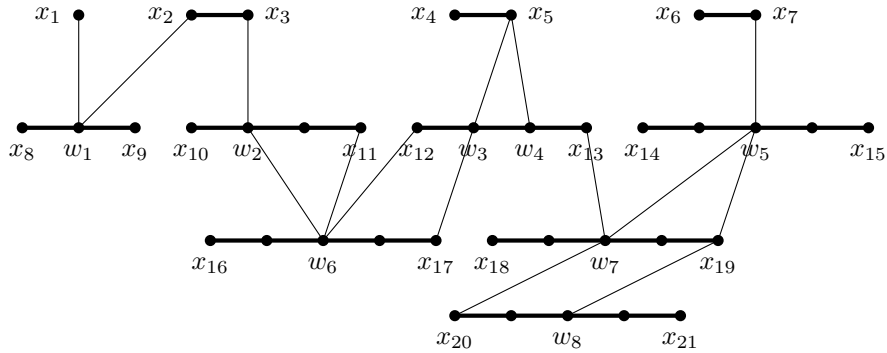
\begin{figure}[h]
\centering
\begin{tikzpicture}
\coordinate (a) at (0.75,4) ;
\coordinate (b) at (2.25,4) ;
\coordinate (c) at (3,4) ;
\coordinate (d) at (5.75,4) ;
\coordinate (e) at (6.5,4) ;
\coordinate (f) at (9,4) ;
\coordinate (g) at (9.75,4);
\coordinate (h) at (0,2.5);
\coordinate (i) at (0.75,2.5);
\coordinate (j) at (1.5,2.5);
\coordinate (k) at (2.25,2.5) ;
\coordinate (l) at (3,2.5) ;
\coordinate (m) at (3.75,2.5) ;
\coordinate (n) at (4.5,2.5) ;
\coordinate (o) at (5.25,2.5) ;
\coordinate (p) at (6,2.5) ;
\coordinate (q) at (6.75,2.5);
\coordinate (r) at (7.5,2.5);
\coordinate (s) at (8.25,2.5) ;

\coordinate (a') at (9,2.5) ;
\coordinate (b') at (9.75,2.5) ;
\coordinate (c') at (10.5,2.5) ;
\coordinate (d') at (11.25,2.5) ;
\coordinate (e') at (2.5,1) ;
\coordinate (f') at (3.25,1) ;
\coordinate (g') at (4,1);
\coordinate (h') at (4.75,1);
\coordinate (i') at (5.5,1);
\coordinate (j') at (6.25,1);
\coordinate (k') at (7,1) ;
\coordinate (l') at (7.75,1) ;
\coordinate (m') at (8.5,1) ;
\coordinate (n') at (9.25,1) ;
\coordinate (o') at (5.75,0) ;
\coordinate (p') at (6.5,0) ;
\coordinate (q') at (7.25,0);
\coordinate (r') at (8,0);
\coordinate (s') at (8.75,0) ;

\draw  (a)--(i)--(b)--(c)--(l)--(g')--(n)--(m)--(l)--(k);
\draw  (h)--(i)--(j);
\draw  (e')--(f')--(g')--(h')--(i');

\draw  (g')--(o)--(p)--(q)--(r)--(l')--(o')--(p')--(q')--(r')--(s');
\draw  (d)--(e)--(p)--(i');
\draw  (e)--(q);
\draw  (j')--(k')--(l')--(m')--(n');
\draw  (q')--(n')--(b');

\draw  (l')--(b')--(g)--(f);
\draw  (s)--(a')--(b')--(c')--(d');
\draw [ultra thick] (b)--(c);
\draw [ultra thick] (d)--(e);
\draw [ultra thick] (f)--(g);
\draw [ultra thick] (h)--(i);
\draw [ultra thick] (i)--(j);
\draw [ultra thick] (k)--(l);
\draw [ultra thick] (l)--(m);
\draw [ultra thick] (m)--(n);
\draw [ultra thick] (o)--(p)--(q)--(r);
\draw [ultra thick] (s)--(a')--(b')--(c')--(d');
\draw [ultra thick] (e')--(f')--(g')--(h')--(i');
\draw [ultra thick] (j')--(k')--(l')--(m')--(n');
\draw [ultra thick] (o')--(p')--(q')--(r')--(s');

\draw (a) node[left=0.1cm]{$x_1$} ;
\draw (b) node[left=0.1cm]{$x_2$} ;
\draw (c) node[right=0.1cm]{$x_3$} ;
\draw (d) node[left=0.1cm]{$x_4$};
\draw (e) node[right=0.1cm]{$x_5$} ;
\draw (f) node[left=0.1cm]{$x_6$} ;
\draw (g) node[right=0.1cm]{$x_7$};
\draw (h) node[below=0.1cm]{$x_8$};
\draw (i) node[below=0.1cm]{$w_1$} ;
\draw (j) node[below=0.1cm]{$x_9$} ;
\draw (k) node[below=0.1cm]{$x_{10}$} ;
\draw (l) node[below=0.1cm]{$w_2$};

\draw (n) node[below=0.1cm]{$x_{11}$} ;
\draw (o) node[below=0.1cm]{$x_{12}$};
\draw (p) node[below=0.1cm]{$w_3$};
\draw (q) node[below=0.1cm]{$w_4$} ;
\draw (r) node[below=0.1cm]{$x_{13}$} ;
\draw (s) node[below=0.1cm]{$x_{14}$} ;

\draw (b') node[below=0.1cm]{$w_5$} ;

\draw (d') node[below=0.1cm]{$x_{15}$};
\draw (e') node[below=0.1cm]{$x_{16}$} ;

\draw (g') node[below=0.1cm]{$w_6$};

\draw (i') node[below=0.1cm]{$x_{17}$} ;
\draw (j') node[below=0.1cm]{$x_{18}$} ;

\draw (l') node[below=0.1cm]{$w_7$};

\draw (n') node[below=0.1cm]{$x_{19}$} ;
\draw (o') node[below=0.1cm]{$x_{20}$};

\draw (q') node[below=0.1cm]{$w_8$};

\draw (s') node[below=0.1cm]{$x_{21}$} ;
\foreach \i in {a,b,c,d,e,f,g,h,i,j,k,l,m,n,o,p,q,r,s,a',b',c',d',e',f',g',h',i',j',k',l',m',n',o',p',q',r',s'}
\draw [fill=black, thick] (\i) circle [radius=0.06];
\end{tikzpicture}
\caption{Illustration of the definitions}
\label{fig:G1}
\end{figure}
${\mathcal{R}}_{1}=\left\{ x_{1}\right\} ,$ ${\mathcal{R}}_{2}=\left\{ \left[
x_{2},x_{3}\right] ,\text{ }\left[ x_{4},x_{5}\right] ,\text{ }\left[
x_{6},x_{7}\right] \right\} ,$ \noindent ${\mathcal{R}}_{3}=\left\{ \left[
x_{8},.,x_{9}\right] \right\} $,\newline
${\mathcal{R}}_{4}=\left\{ \left[ x_{10},..,x_{11}\right] ,\left[
x_{12},..,x_{13}\right] \right\} ,$\newline
${\mathcal{R}}_{5}=\left\{ \left[ x_{14},..,x_{15}\right] ,\left[
x_{16},..,x_{17}\right] ,\left[ x_{18},..,x_{19}\right] ,\left[
x_{20},..,x_{21}\right] \right\} .$\newline
$End({\mathcal{R}}_{3})=\left\{ x_{8},x_{9}\right\} ,End({\mathcal{R}}%
_{4})=\left\{ x_{10},x_{11},x_{12},x_{13}\right\} ,$\newline
$End({\mathcal{R}}_{5})=\left\{
x_{14},x_{15},x_{16},x_{17},x_{18},x_{19},x_{20},x_{21}\right\} .$\newline
$C_{3}=\left\{ w_{3}\right\} ,$ $C_{5}=\left\{
w_{5},w_{6},w_{7},w_{8}\right\} .$\smallskip \newline
\end{example}

For $x\in End(R),$ $N_{ext}\left( x\right) $ is the set of non path
neighbors (neighbors of $x$ outside its own path $R$) and $N_{ext}(X^{\prime
})=\bigcup\limits_{x\in X^{\prime }}N_{ext}(x)$ with $X^{\prime }\subset End(%
\mathcal{P}).$\text{\ } Now using $N_{ext}(V_{1}\cup V_{2})$, we define a subset $X
$ of $End(\mathcal{P})$ and we denote $N_{ext}(X)$ by $W.$ \noindent Let $%
X_{1}=V_{1}\cup V_{2},$ $W_{1}=N_{ext}\left( X_{1}\right) $ and for $t\geq 1,
$ $X_{t}$ being defined, let 
\begin{equation}
X_{t+1}=X_{t}\cup (\bigcup\limits_{N_{ext}(X_{t})\cap Int(R)\neq \emptyset ,R\in 
\mathcal{R}}End(R)).
\end{equation}
Let $s\geq 1$ the first integer such that $X_{s}=X_{s+1}$. Let
us set $X=X_{s},$ $W=N_{ext}\left( X\right) $ and for $t\in \{1,...,s\},$
let $W_{t+1}=N_{ext}\left( X_{t+1}\right) \backslash N_{ext}\left(
X_{t}\right) .$ Then $W=\bigcup\limits_{i=1}^{i=s}W_{i}.$ Here is an
example of that construction.

\begin{example}
For the partition in Figure \ref{fig:G1}, we have $X_{1}=%
\{x_{1},x_{2},...,x_{7}\}$, \newline
$X_{2}=X_{1}\cup \{x_{8},x_{9},...,x_{15}\}$, $X_{3}=X_{2}\cup
\{x_{16},x_{17},x_{18},x_{19}\}$, $X_{4}=X_{3}\cup \{x_{20},x_{21}\}=X$, $%
W_{1}=\{w_{1},w_{2},...,w_{5}\}$, $W_{2}=\{w_{6},w_{7}\}$, $W_{3}=\{w_{8}\}$
and $W=\{w_{1},w_{2},...,w_{8}\}$.\smallskip \newline
\end{example}

Let $X_{0}=\emptyset .$ Pick $w_{r}\in W_{r}$ for some $r$. By definition of 
$w_{r}$, there exists a sequence 
\begin{equation}
\alpha \left( w_{r}\right) =x_{1}w_{1},x_{2}w_{2},...,x_{r}w_{r},
\end{equation}%
where for each $t\in \left\{ 1,...,r\right\} $, $x_{t}\in X_{t}-X_{t-1}$,
$w_t\in W_t$
and $x_{t}w_{t}$ is an edge joining two paths of the partition. In addition,
for each $t\in \left\{ 1,...,r-1\right\} $, $w_{t}$ and $x_{t+1}$ are in the
same path of the partition. $\bigskip $

\emph{The sequence }$\alpha (w_{r})$\emph{\ has a good order} if the vertex $%
w_{r}$ belongs to a path $R$ with end vertices, say $x_{r+1}$ and $%
x_{r+1}^{\prime }$, in $X_{r+1}-X_{r}.$ \emph{The vertex }$w_{r}$\emph{\ is
then said to be of good order}. Using a sequence $\alpha (w_{r})$ with \emph{good
order}, we can define two new partitions as follows.\medskip

For each $i\in \{1,...,r+1\},$ we orient the paths of $\mathcal{P}$ such
that each $x_{i}$ is the terminal extremity. We denote\ by\ $w_{t}^{+}$ and $%
w_{t}^{-}$ the successor and the predecessor of $w_{t},$ respectively.%
\medskip

1) $\mathcal{P}_{1}\left( w_{r}\right) $ is obtained from $\mathcal{P}$ by
deleting the edges $w_{t}w_{t}^{+},$ $1\leq t\leq r$ and adding the edges $%
x_{t}w_{t}$ for $1\leq t\leq r;$

2) $\mathcal{P}_{2}\left( w_{r}\right) $ is obtained from $\mathcal{P}$ by
deleting the edges $w_{t}w_{t}^{+},$ $1\leq t\leq r-1$ and the edge $%
w_{r}w_{r}^{-}$ and adding the edges $x_{t}w_{t}$ for $1\leq t\leq r$%
.\bigskip \newline
If we consider the sets of edges of these partitions we note that\ 
\begin{equation*}
E(\mathcal{P}_{2})=(E(\mathcal{P}_{1})-w_{r}w_{r}^{-})\cup w_{r}w_{r}^{+}.
\end{equation*}%
Furthermore, $|\mathcal{P}_{2}|=|\mathcal{P}_{1}|=\mu (G).$\newline
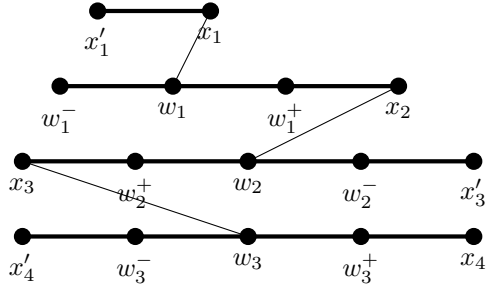
\begin{figure}[h]
\centering
\begin{tikzpicture}
\coordinate (a) at (-1,2) ;
\coordinate (b) at (0.5,2) ;
\coordinate (c) at (-1.5,1) ;
\coordinate (d) at (0,1) ;
\coordinate (e) at (1.5,1) ;
\coordinate (f) at (3,1) ;
\coordinate (g) at (-2,0);
\coordinate (h) at (-0.5,0);
\coordinate (i) at (1,0);
\coordinate (j) at (2.5,0) ;
\coordinate (k) at (4,0) ;
\coordinate (l) at (-2,-1) ;
\coordinate (m) at (-0.5,-1) ;
\coordinate (n) at (1,-1) ;
\coordinate (o) at (2.5,-1) ;
\coordinate (p) at (4,-1) ;

\draw (a) node[below=0.1cm]{$x'_1$} ;
\draw (b) node[below=0.1cm]{$x_1$} ;
\draw (c) node[below=0.1cm]{$w^-_1$} ;
\draw (d) node[below=0.1cm]{$w_1$};
\draw (e) node[below=0.1cm]{$w^+_{1}$} ;
\draw (f) node[below=0.1cm]{$x_{2}$} ;
\draw (g) node[below=0.1cm]{$x_3$};
\draw (h) node[below=0.1cm]{$w^+_{2}$};
\draw (i) node[below=0.1cm]{$w_2$} ;
\draw (j) node[below=0.1cm]{$w^-_2$} ;
\draw (k) node[below=0.1cm]{$x'_3$} ;
\draw (l) node[below=0.1cm]{$x'_4$};
\draw (m) node[below=0.1cm]{$w^-_{3}$} ;
\draw (n) node[below=0.1cm]{$w_{3}$} ;
\draw (o) node[below=0.1cm]{$w^+_3$};
\draw (p) node[below=0.1cm]{$x_{4}$};

\draw  (a)--(b)--(d);
\draw  (c)--(d)--(e)--(f)--(i);
\draw  (g)--(h)--(i)--(j)--(k);
\draw  (g)--(n);
\draw  (l)--(m)--(n)--(o)--(p);

\draw [ultra thick] (a)--(b);
\draw [ultra thick] (c)--(d)--(e)--(f);
\draw [ultra thick] (g)--(h)--(i)--(j)--(k);
\draw [ultra thick] (l)--(m)--(n)--(o)--(p);
\foreach \i in {a,b,c,d,e,f,g,h,i,j,k,l,m,n,o,p}
\draw [fill=black, thick] (\i) circle [radius=0.1];
\end{tikzpicture}
\caption{Graph with $\protect\mu (G)=4$}
\label{fig:G2}
\end{figure}
For example, the sequence $\alpha (w_{2})$ in the graph of Figure \ref%
{fig:G2}, defines two partitions.\\
We have $\mathcal{P}_{1}\left(
w_{2}\right) =\left\{ x_{1}^{\prime
}x_{1}w_{1}w_{1}^{-},w_{1}^{+}x_{2}w_{2}w_{2}^{-}x_{3}^{\prime
},x_{3}w_{2}^{+},x_{4}^{\prime }w_{3}^{-}w_{3}w_{3}^{+}x_{4}\right\} $ and $%
\mathcal{P}_{2}\left( w_{2}\right) =\left\{ x_{1}^{\prime
}x_{1}w_{1}w_{1}^{-},w_{1}^{+}x_{2}w_{2}w_{2}^{+}x_{3},w_{2}^{-}x_{3}^{%
\prime },x_{4}^{\prime }w_{3}^{-}w_{3}w_{3}^{+}x_{4}\right\} .\smallskip $%
\newline

 We denote by $R_{i}[x^{\prime },x]$ any path of order $i$
oriented from $x^{\prime }$ to $x.$ So $x^{\prime }$ is the initial end of $%
R_{i}$ and $x$ is its terminal end.

\begin{ob}
\label{ob 1}If $w\in W_{t},$ then for some $i,$ $w$ belongs to
some path $R_{i}\left[ x^{\prime },x\right] .$  The path $R_{i_{1}}%
\left[ w^{+},x\right] $ is in $\mathcal{P}_{1}\left( w\right) $ and the path 
$R_{i_{2}}\left[ x^{\prime },w^{-}\right] $ is in $\mathcal{P}_{2}\left(w\right) .$
Note that the subpath $R_{i_{1}}\left[ w^{+},x\right] $ (resp. $R_{i_{2}}%
\left[ x^{\prime },w^{-}\right] )$ is of order $i_{1}$ (resp. $i_{2})$ such
that $i_{1}+i_{2}+1=i.\bigskip $
\end{ob}

\section{Proof of Theorem \protect\ref{thm 1} :\protect\medskip}

We choose a minimum path-partition $\mathcal{P}_{0}$ such that\newline
1) $p_{1}$ is minimum, \newline
2) if (1) is satisfied, then $p_{2}$ is minimum. \medskip\ \newline
Let $\mathcal{P}^{\prime }$ be the set of paths with end vertices in $X$.
Let $p=|\mathcal{P}^{\prime }|.$ Let $p_{i}$ be the number of paths of order 
$i$ in $\mathcal{P}^{\prime }.$\newline

Sketch of the proof\newline
In view to bound $\mu (G)$ we want to bound $p_{1}$ and $p_{2}$. We consider
the set $X$ generated by $V_{1}\cup V_{2}$, and therefore the two sets $%
W=N_{ext}(X)$ and $\varepsilon (X,W)$. Note that the cardinality of $X$
is $2p-p_{1}$, The proof of our theorem is done through the following
steps. We want to bound in two manners the number of edges $\varepsilon
(X,W).$ The upper bound will use $W$ and $\Delta $, the lower bound will use 
$X$ and $\delta $.\newline
In the first part of the proof, we show some claims relative to the set $W$
and one relative to the lower bound of $\varepsilon (x,W)$ for $x\in X.$
In the second part of the proof, we calculate the bounds of $\varepsilon (X,W)$. We get finally an upper bound for $p_{1}+2p_{2}$ in
function of $p,\delta $ and $\Delta $, and, then an upper bound for $\mu (G)$%
.

\bigskip \textbf{A) Claims}

\begin{claim}
\label{cl:1} \mbox{} 

\begin{enumerate}
\item For each $v\in V_{1},$ $N(v)\subset C_{3}$.

\item For each $a\in V_{2},$ $N(a)\subset C_{3}\cup Int(\mathcal{R}_{4})\cup
C_{5}.$ 
\end{enumerate}

So, $\left\{ 
\begin{array}{c}
\left\vert N\left( a\right) \cap R\right\vert \leq 1,\text{ for every path }R%
\text{ of order }3\text{ or }5. \\ 
\left\vert N\left( a\right) \cap R\right\vert \leq 2,\text{ \ \ \ \ \ \ for
every path }R\text{ of order }4.%
\end{array}%
\right. .$
\end{claim}

\begin{proof}[Proof of Claim \ref{cl:1}]$ $\newline
1) Let $v$ be a vertex of $V_{1}$. By the minimality of $\mathcal{P}_{0}$, $v
$ is not adjacent to an end vertex of another path in $\mathcal{P}_{0}$. Let 
$w$ be a neighbor of $v$ in a path 
oriented from $x^{\prime }$ to $x.$ We have two partitions.
In $\mathcal{P}_{0}$, we replace the path $v$ and the path $R_{i}\left[
x^{\prime },x\right] $ either by the pair of paths $vR_{i_{1}}\left[ w,x%
\right] ,$ $R_{i_{2}}\left[ x^{\prime },w^{-}\right] $ or by the pair of
paths $R_{i_{1}^{\prime }}\left[ x^{\prime },w\right] v,$ $R_{i_{2}^{\prime
}}\left[ w^{+},x\right] .$ By the minimality of $p_{1}$, w is both  predecessor of x and successor of $x^{\prime}$. So the order of $R_{i}\left[
x^{\prime },x\right] $ is $3,$ and $w$ is the center of $R_{i}\left[
x^{\prime },x\right] $. Thus $N\left( v\right) \subset C_{3}.$\newline
2) Let $w$ be a neighbor of $a$ in $R_{i}\left[ x^{\prime },x\right] .$
As precedently, we get two partitions and by definition of $\mathcal{P}_{0}$, each of $%
R_{i}\left[ x^{\prime },w^{-}\right] $ and $R_{i}\left[ w^{+},x\right] $
should be of order at most two. So $N\left( a\right) \subset C_{3}\cup
C_{5}\cup Int\left( \mathcal{R}_{4}\right) ,$ completing the
proof of Claim \ref{cl:1} 
\end{proof}

\begin{claim}
\label{cl:2} Let $W_{a}$ be the set of vertices of good order in $W.$ Then

\begin{enumerate}
\item $W_{a}\subset C_{3}\cup Int\left( \mathcal{R}_{4}\right) \cup C_{5}.$

\item $W=W_{a}.$
\end{enumerate}
\end{claim}

\begin{proof}[Proof of Claim \ref{cl:2}]$ $\newline
1) Suppose that there exists $w\in W_{a}$ such that $w$ is in the path $R_{i}%
\left[ x^{\prime },x\right] .$ By Observation \ref{ob 1}, we have $%
i-1=i_{1}+i_{2}$. If $i_{1}\geq 3$ (resp. $i_{2}\geq 3)$, then $\mathcal{P}%
_{1}\left( w\right) $ (resp. $\mathcal{P}_{2}\left( w\right) $) contains $%
p_{1}-1$ paths of order $1$ or $p_{1}$ paths of order $1$ and $p_{2}-1$
paths of order $2$. A contradiction with the definition of $\mathcal{P}_{0}.$
Thus $W_{a}\subset C_{3}\cup Int\left( \mathcal{R}_{4}\right) \cup C_{5}.$
\newline

2) Suppose that $W\neq W_{a}.$ In $W_{b}=W\backslash W_{a}$ there exists
necessarely a vertex $w=w_{r}$ with sequence $\alpha \left(
w\right) =x_{1}w_{1},x_{2}w_{2},...,x_{r}w_{r}$ with $x_{t}\in X_{t}-X_{t-1}$
 $x_{r}$ is an end vertex of some path $R=\left[ x_{r}^{\prime },...,x_{r}%
\right] .$ By the definition of $w$, the vertex $w$ belongs to a path $R^{\prime }= [x_{j}^{\prime },...,x_{j} ]$ with $x_{j}\in X_{j},$
$j\leq r.$ By the definition of $X_{j},$ the path $R^{\prime }$
contains one element of $W_{a}$, say $w_{a}.$ By the definition of $W_{a}$, $w_{a}$ is
adjacent to a vertex $x^{\prime \prime }$ of $X_{j-1},$ end vertex of a path 
$R^{\prime \prime }$. Since $W_{a}\subset C_{3}\cup Int\left( \mathcal{R}_{4}
\right) \cup C_{5},$ then $\left\vert R^{\prime }\right\vert =4$ or $5$
and $w=w_{a}^{-}$ or $w=w_{a}^{+}.$ The end vertex $x_{r}\in X_{r}$ and the
end vertex $x^{\prime \prime }\in X_{j-1}$ are adjacent respectively to $w$
and $w_{a}$, successive vertices of the same path $R^{\prime }$.

\begin{figure}[h]
\centering
\begin{tikzpicture}
\coordinate (a) at (-1,0) ;
\coordinate (b) at (0,0) ;
\coordinate (c) at (1,0) ;
\coordinate (d) at (2,0) ;
\coordinate (e) at (3,0) ;
\coordinate (f) at (2,-2) ;
\coordinate (g) at (3,-2);
\coordinate (h) at (4,-2);
\coordinate (i) at (5,-2);
\coordinate (j) at (4,0) ;
\coordinate (k) at (5,0) ;
\coordinate (l) at (6,0) ;
\coordinate (m) at (7,0) ;
\coordinate (n) at (8,0) ;
\draw (a) node[above=0.1cm]{$x'_r$} ;
\draw (e) node[above=0.1cm]{$x_r$} ;
\draw (f) node[below=0.1cm]{$x'_j$} ;
\draw (i) node[below=0.1cm]{$x_j$};
\draw (j) node[above=0.1cm]{$x'_{j-1}$} ;
\draw (n) node[above=0.1cm]{$x''$} ;
\draw (g) node[below=0.1cm]{$w$};
\draw (h) node[below=0.1cm]{$w_{a}$};
\draw (1,0.7) node {$ R $};
\draw (6,0.7)node {$R'' $};
\draw (3.5,-2.7) node {$R'$};
\draw  (a)--(b)--(c)--(d)--(e);
\draw  (j)--(k)--(l)--(m)--(n);
\draw  (f)--(g)--(h)--(i);
\draw  (e)--(g);
\draw  (h)--(n);
\foreach \i in {a,b,c,d,e,f,g,h,i,j,k,l,m,n}
\draw [fill=black, thick] (\i) circle [radius=0.1];
\end{tikzpicture}
\caption{Paths $R$, $R'$, $R''$}
\label{fig:G3}
\end{figure}
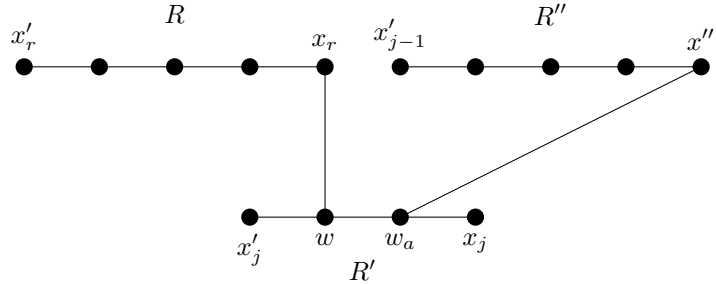

We get a partition with $p-1$ paths. We replace the three paths $R,$ $%
R^{\prime }$ and $R^{\prime \prime }$ by the two paths composed by $R\cup
R^{\prime }\cup R^{\prime \prime }\cup \{x_{r}w,x''w_{a}\}-~\{ww_{a}\}$
(see Figure \ref{fig:G3}), a contradiction with the minimality of $\mathcal{P}%
_{0}.$ Thus, $W_{b}=\emptyset $ and so, $W=W_{a}.$ This completes the proof of Claim \ref{cl:2}. 
\end{proof}

\begin{claim}
\label{cl:3} For each path $R$ of order $4$ in $\mathcal{P}%
_{0}, $ we have $\left\vert W\cap V\left( R\right) \right\vert \leq 2.$
Furthermore if $\left\vert W\cap V\left( R\right) \right\vert =2$, then
there is a unique $x\in X$ such that $W\cap V\left( R\right) \neq \emptyset $
and thus $\varepsilon (X,R)=4$
\end{claim}

\begin{proof}[Proof of Claim \ref{cl:3}]$ $\newline
By the minimality of $\mathcal{P}_{0}$ for each path $R\in \mathcal{R}_{4}$,
we have $\left\vert W\cap V\left( R\right) \right\vert \leq 2$. Now, assume
that there exists a path $R\in \mathcal{R}_{4}$ such that $W\cap V\left(
R\right) =\left\{ w_{1},w_{2}\right\} $ where $w_{i}\in N_{ext}\left(
x_{i}\right) ,$ $i=1,$ $2.$ By taking off the edge $w_{1}w_{2}$ and adding
the edges $x_{1}w_{1}$ and $x_{2}w_{2}$ we obtain a partition with $p-2$
paths, a contradiction. So $x_{1}=x_{2}.$ Let $R^{\prime }(x_{1})$ be the
path of extremity $x_{1}$ in the partition. We may suppose $R=[x_{0}^{\prime
},w_{1},w_{2},x_{0}]$. If $x_{0}$ is neighbor of $w_{1}$, then we replace
the two paths $R$ and $R^{\prime }(x_{1})$ by the path $R^{\prime
}(x_{1})\cup \lbrack w_{2}, x_{0},w_{1},x'_{0}]$. We get a partition with $p-1$
paths, a contradiction. So, there is no edge $x_{0}w_{1}$. Similarly, there
is no edge $x_{0}^{\prime }w_{2}$. It follows that $\varepsilon (X,R)=4$, {%
 completing the proof of Claim \ref{cl:3}} 
\end{proof}
For $i\in \left\{ 1,2\right\} ,$ let $\mathcal{R}_{4}^{i}$ be the set of
paths of order $4$, which contain exactly $i$ elements of $W$.\newline
For the lower bound we shall need the following claim.

\begin{claim}
\label{cl:4} \mbox{}

\begin{enumerate}
\item If $x\in X\cap V(\mathcal{R}_{1}\cup \mathcal{R}_{3})$, then $x$ has $%
d(x) $ neighbors in $W$.

\item If $x_{1},$ $x_{2}\in X\cap V(\mathcal{R}_{4}),$ then the
set $\{x_{1},x_{2}\}$ has $d(x_{1})+d(x_{2})-1$ neighbors in $W$ if $x_{1},$ 
$x_{2}$ belong to ${V}\left( {\mathcal{R}_{4}^{1}}\right) ${\ and it has $%
d(x_{1})+d(x_{2})$ neighbors if $x_{1},$ $x_{2}$ belong to }${V}\left( {%
\mathcal{R}_{4}^{2}}\right) $.

\item If $x\in X\cap V(\mathcal{R}_{2}\cup \mathcal{R}_{5})$, then $x$ has $%
d(x)-1$ neighbors in $W$.
\end{enumerate}
\end{claim}

\begin{proof}[Proof of Claim \ref{cl:4}]$ $\newline
1) If $x\in V_{1},$ then by Claim $1$, $N\left( x\right) \subset W.$ If $%
x\in X\cap V\left( \mathcal{R}_{3}\right) ,$ then let $R=xyx^{\prime }.$ By
the minimality of $\mu $, $x$ and $x^{\prime }$ are not adjacent. It follows
that $N\left( x\right) \subset $$W$.
2) Let $R^{\prime }=x_{1}yzx_{2}.$ By the minimality of $p_{2},$ 
there is no edge $x_{1}x_{2}.$\newline
If $W\cap  R' =\left\{ y\right\} $ then $%
x_{1}z\notin E\left( G\right) $ by the minimality of the path partition. It
follows that $N\left( x_{1}\right) \subset $$W$ and $(N\left( x_{2}\right)
-\left\{ z\right\} )\subset W$.\newline
If $W\cap V\left( R^{\prime }\right) =\left\{ y,z\right\} ,$ then $N\left(
x_{1}\right) \subset W$ and $N\left( x_{2}\right) \subset W.$\newline
3) Clearly if $x\in V_{2}$, then $x$ has $d\left( x\right) -1$ neighbors in $%
W.$\newline
Let $R^{\prime \prime }=x_{1}yztx_{2}.$ By Claim \ref{cl:2}, we have
$W\cap V\left( R^{\prime\prime }\right) =\left\{ z\right\} .$ It follows that
$N\left( x_{1}\right) -\left\{y\right\} \subset w$ and $N\left( x_{2}\right) -\left\{ t\right\} \subset W,$ completing the proof of Claim \ref{cl:4}. 
\end{proof}
Now we bound first $|\mathcal{P}^{\prime }|$ then $\mu (G).$\newline

B) \textbf{Calculations of the bounds}\newline
1) Bound of $p_{1}+2p_{2}$ \newline
Let $k=\frac{\Delta }{\delta }.$ We shall prove the following inequality.

\begin{claim}
\label{cl:5} $p_{1}+2p_{2}\leq (p_{3}+p_{4}+p_{5})(k-2)+\frac{2}{\delta }%
p_{2}$ where $p_{i}$ is the number of paths of order $i$ in $\mathcal{P}%
^{\prime }$.
\end{claim}

\begin{proof}[Proof of Claim \ref{cl:5}]$ $\newline
Put $p_{_{4}}^{\prime }=\left\vert \mathcal{R}_{4}^{1}\right\vert $\ and $%
p_{_{4}}^{\prime \prime }=\left\vert \mathcal{R}_{4}^{2}\right\vert .$ Let $%
\varepsilon \left( X,W\right) $ be the number of edges between $X$ and $W$.
Let $w$ be any vertex of $W$.
Observe that for $w\in $$ Int\left( \mathcal{R}_{5}\right) $, $w$ has at
most\ $\Delta -2$\ neighbors in $X.$ For $w\in Int\left( 
\mathcal{R}_{4}^{1}\right) ,$\ then $w$ has at least one neighbor which
does not belong to $X$ and so, $w$ has at most $\Delta -1$ neighbors in $X.$
By Claim \ref{cl:3}, if $w\in Int\left( \mathcal{R}_{4}^{2}\right) ,$
then $w$ has exactly two neighbors in $X.$ If $w\in $$W\cap  Int{%
R}_{3}$ then $w$ has at most $\Delta $ neighbors in $X.$ It follows
that $\varepsilon \left( X,W\right) \leq p_{3}\Delta +p_{4}^{\prime }\left(
\Delta -1\right) +4p_{4}^{\prime \prime }+p_{5}\left( \Delta -2\right) .$

On the other hand, by Claim \ref{cl:4},\newline
$\varepsilon \left( X,W\right) \geq \sum\limits_{x\in End(\mathcal{R}%
_{1})}d(x)+\sum\limits_{x\in End(\mathcal{R}_{2})}(d(x)-1)+\sum\limits_{x\in
End(\mathcal{R}_{3})\cap X}d(x)$\newline
${+}$$\sum\limits_{x_{1},x_{2}\in End(\mathcal{R}_{4}^{1})\cap
X}(d(x_{1})+d(x_{2})-1)+\sum\limits_{x_{1},x_{2}\in End(\mathcal{R}%
_{4}^{2})\cap X}(d(x_{1})+d(x_{2}))$\newline
$+\sum\limits_{x\in End(\mathcal{R}_{5})\cap X}(d(x)-1).$\ $\bigskip $%
\newline
As for each $x\in X$ we have $\delta \leq d(x)$, it follows that\newline
$p_{_{1}}\delta +2p_{_{2}}(\delta -1)+2p_{_{3}}\delta +%
p_{_{4}}^{\prime }(2\delta -1)+p_{4}^{\prime \prime }2\delta %
+2p_{_{5}}(\delta -1)\leq p_{3}\Delta +p_{4}^{\prime }\left( \Delta
-1\right) +4p_{4}^{\prime \prime }+p_{5}\left( \Delta -2\right) $.\newline
As $k\geq 2$ and $\delta $ $\geq 2,$ then $\Delta \geq 4$ and we replace $%
4p_{4}^{\prime \prime }$ by $\Delta p_{4}^{\prime \prime }$ in the last
inequality. Then 
$p_{_{1}}\delta +2p_{_{2}}(\delta -1)+2p_{_{3}}\delta +p_{_{4}}^{\prime
}(2\delta -1)+p_{_{4}}^{\prime \prime }2\delta +2p_{_{5}}(\delta -1)\leq
p_{3}\Delta +p_{4}^{\prime }\left( \Delta -1\right) +p_{4}^{\prime \prime
}\Delta $$+p_{5}\left( \Delta -2\right) .$ \newline
As $p_{_{4}}=p_{_{4}}^{\prime }+p_{_{4}}^{\prime \prime }$, we get $p_{_{1}}\delta +2p_{_{2}}(\delta -1)\leq p_{3}\left( \Delta -2\delta
\right) +p_{4}\left( \Delta -2\delta \right) +p_{5}\left( \Delta -2\delta
\right) .$\newline
Since $\Delta =k\delta ,$ then 
\begin{equation*}
p_{1}+2p_{2}\leq (p_{3}+p_{4}+p_{5})(k-2)+\frac{2}{\delta }p_{2}.
\end{equation*}%
This completes the proof of Claim \ref{cl:5} 
\end{proof}

{2) Calculation of the bound of }$\left\vert \mathcal{P}^{\prime
}\right\vert =p$\newline
By Claim \ref{cl:5}, there exists $r\leq k-2,$ such that 
\begin{equation}
p_{_{1}}+2p_{_{2}}=r\left( p_{_{3}}+p_{_{4}}+p_{_{5}}\right) +\frac{2}{%
\delta }p_{_{2}}  \label{un}
\end{equation}%
Let us call $n_{1}$ the order of $V(\mathcal{P}^{\prime })$. Clearly 
\begin{equation*}
p=p_{_{1}}+p_{_{2}}+p_{_{3}}+p_{_{4}}+p_{_{5}}
\end{equation*}%
and 
\begin{equation*}
n_{1}=p_{_{1}}+2p_{_{2}}+3p_{_{3}}+4p_{_{4}}+5p_{_{5}}.
\end{equation*}%
Using equality (\ref{un}), we obtain 
\begin{equation*}
p=\left( r+1\right) \left( p_{_{3}}+p_{_{4}}+p_{_{5}}\right) +\left( \frac{2%
}{\delta }-1\right) p_{_{2}}.
\end{equation*}%
As $\delta \geq 2,$ we get $p\leq \left( r+1\right) \left(
p_{_{3}}+p_{_{4}}+p_{_{5}}\right) .$ Again by equality (\ref{un}), we have%
\newline
$n_{1}=\left( r+3\right) p_{_{3}}+\left( r+4\right) p_{_{4}}+\left(
r+5\right) p_{_{5}}+\frac{2}{\delta }p_{_{2}}.$ This yields 
\begin{equation*}
n_{1}\geq \left( r+3\right) (p_{_{3}}+p_{_{4}}+p_{_{5}}).
\end{equation*}%
Since $r\leq \left( k-2\right) ,$ we get the inequality $p\leq \frac{k-1}{k+1%
}n_{1}.$ \newline
3) Bound of $\mu (G)$\newline
Let $G_{2}=G-V(\mathcal{P}^{\prime }).$ Let $n_{2}=n-n_{1}.$ Clearly $\mu
(G)\leq p+\mu (G_{2}).$ We know that $p\leq \frac{k-1}{k+1}n_{1}.$ Recall
that each path of $\mathcal{P}_{0}$ contained in $G_{2}$ has order at least $%
3$. It follows that $\mu (G_{2})\leq \frac{n_{2}}{3}.$ Since $k\geq 2,$ we
have $\frac{1}{3}\leq \frac{k-1}{k+1}$ and so $\mu (G_{2})\leq \frac{k-1}{k+1%
}n_{2}.$ Thus $\mu (G)\leq \frac{k-1}{k+1}n.$ This finishes the proof of the
theorem.$\blacksquare $


\bf{Declaration}

\it
The authors declare no funds no grants were received during the preparation of the manuscript\\
The authors declare they have no interest to disclose.\\
All the authors contribuated to this work and they approve the final manuscript.

\end{document}